\documentclass[epsf,10pt]{article}
\usepackage{enumerate}
\usepackage{epsfig}
\usepackage{amsmath,amssymb,latexsym}
\usepackage{amsthm}
\usepackage{graphicx} 
\usepackage{tikz}
\usepackage{float}
\newcommand\GG{\mathcal{G}}
\newcommand\NN{\mathbf{N}}

\newcommand{\address}{Mathematics Department, P.O. Box 311430, Denton, TX 76203-1430; e-mail:jonathanleung@my.unt.edu}

\newtheorem{theorem}{Theorem}[section]
\newtheorem{definition}[theorem]{Definition}
\newtheorem{corollary}[theorem]{Corollary}

\newtheorem{lemma}[theorem]{Lemma}

\title{Shifted L\'evy's Dragon Curve and Directed Graph}
\author{Jonathan Leung \\University of North Texas \footnote{\address}}
\date{\today}

\begin{document}
\maketitle

\renewcommand{\thefootnote}{\ }
\begin{abstract}

It is known that every point on L\'evy’s Dragon Curve admits a natural representation as a complex power series. We introduce a directed graph $\GG_1$ which characterizes this representation.
In this paper, we study the translation of the curve by $s=-1/2+i/2$. 
We identify another directed graph $\GG_2$, that characterizes the translated curve and exhibits a revolving structure analogous to that of $\GG_1$. 

\footnote{This paper is based on an undergraduate honors thesis at University of North Texas}.

\end{abstract}

\section{Introduction}

Self-similar fractals are structures that replicate their overall form at every scale, revealing intricate patterns found throughout nature, art, and mathematics. In this paper, we explore a particularly striking example: Lévy’s Dragon Curve—a self-similar curve notable for its ability to tile the complex plane. 

L\'evy's Dragon Curve $L$ (Figure \ref{fig:dragon}) was introduced and studied in 1938 by P. L\'evy \cite{Levy-1938}. It is the unique attractor $L$ in the complex plane satisfying the set equation: $L=\psi_0(L)\cup\psi_1(L)$,
where
\begin{equation}
\label{eq:levy}
\begin{cases}
 \psi_0(z)=(\frac{1-i}{2})z,& \\
 \psi_1(z)=(\frac{1+i}{2})z + \frac{1-i}{2}. &
\end{cases}
\end{equation}

Figure \ref{fig:construction} shows how $L$ can be constructed by similar contractions $\{\psi_0, \psi_1 \}$ from the isosceles triangle $A_0$ with vertices $0,1$, and $\tfrac12-\tfrac12 i$. It is easy to see from this recursive construction that the area of $L$ is the same as that of $A_0$. In fact, L\'evy proved that $L$ is a space-filling curve.

\begin{figure}[H]
  	\begin{center}
	\epsfig{file=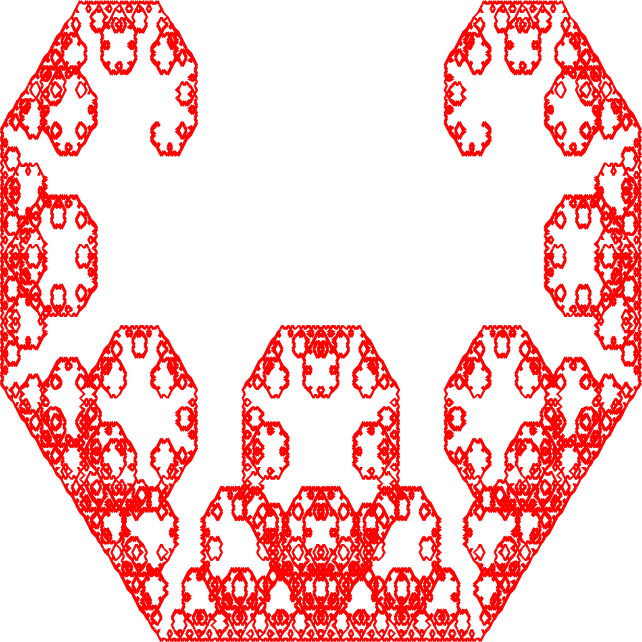, height=1.75in, width=.51\textwidth} 
  \caption{L\'evy's dragon curve}
	\end{center}
\label{fig:dragon}
\end{figure}

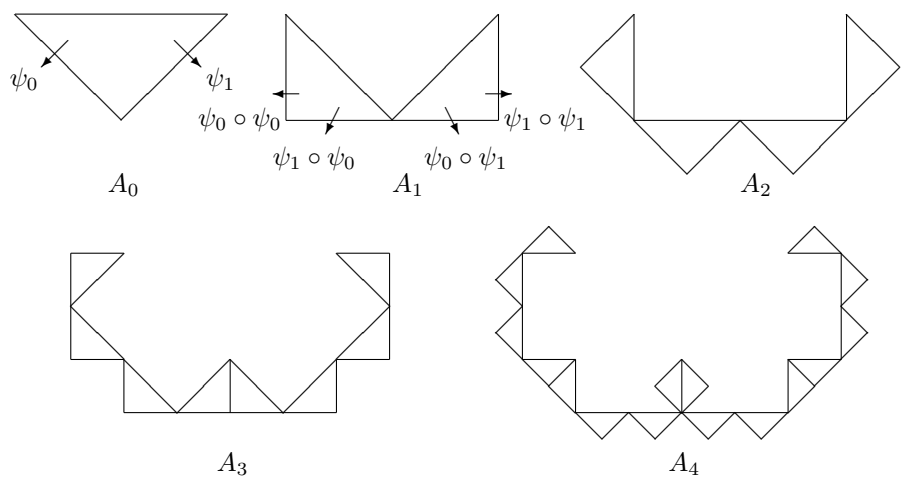
\begin{figure}[H]
\begin{center}
\begin{picture}(320,180)(20,-120)
	\put(14,50){\line(1,-1){40}}
	\put(14,50){\line(1,0){80}}
	\put(54,10){\line(1,1){40}}
	\put(34,40){\vector(-1,-1){10}}
	\put(74,40){\vector(1,-1){10}}
	\put(12,30){\makebox(0,0)[tl]{$\psi_0$}}
	\put(86,30){\makebox(0,0)[tl]{$\psi_1$}}	
	\put(49,-10){\makebox(0,0)[tl]{$A_0$}}
	\put(116,50){\line(1,-1){40}}
	\put(156,10){\line(1,1){40}}
	\put(116,10){\line(0,1){40}}
	\put(116,10){\line(1,0){80}}
	\put(196,50){\line(0,-1){40}}
	\put(121,20){\vector(-1,0){10}}
	\put(83,15){\makebox(0,0)[tl]{$\psi_0 \circ \psi_0$}}
	\put(136,15){\vector(-1,-2){5}}
	\put(111,0){\makebox(0,0)[tl]{$\psi_1 \circ \psi_0$}}
	\put(176,15){\vector(1,-2){5}}
	\put(168,0){\makebox(0,0)[tl]{$\psi_0 \circ \psi_1$}}
	\put(191,20){\vector(1,0){10}}
	\put(198,15){\makebox(0,0)[tl]{$\psi_1 \circ \psi_1$}}
	\put(156,-10){\makebox(0,0)[tl]{$A_1$}}
	\put(247,10){\line(0,1){40}}
	\put(247,10){\line(1,0){80}}
	\put(327,50){\line(0,-1){40}}
	\put(287,-10){\makebox(0,0)[tl]{$A_2$}}
	\put(247,10){\line(-1,1){20}}
	\put(227,30){\line(1,1){20}}
	\put(247,10){\line(1,-1){20}}
	\put(267,-10){\line(1,1){20}}
	\put(287,10){\line(1,-1){20}}
	\put(307,-10){\line(1,1){20}}
	\put(327,50){\line(1,-1){20}}
	\put(347,30){\line(-1,-1){20}}
	\put(90,-115){\makebox(0,0)[tl]{$A_3$}}
	\put(35,-60){\line(1,1){20}}
	\put(55,-80){\line(-1,1){20}}
	\put(55,-80){\line(1,-1){20}}
	\put(75,-100){\line(1,1){20}}
	\put(95,-80){\line(1,-1){20}}
	\put(115,-100){\line(1,1){20}}
	\put(135,-40){\line(1,-1){20}}
	\put(155,-60){\line(-1,-1){20}}
	\put(55,-80){\line(-1,0){20}}
	\put(35,-80){\line(0,1){40}}
	\put(35,-40){\line(1,0){20}}
	\put(55,-100){\line(0,1){20}}
	\put(55,-100){\line(1,0){80}}
	\put(95,-80){\line(0,-1){20}}
	\put(135,-80){\line(0,-1){20}}
	\put(135,-80){\line(1,0){20}}
	\put(155,-40){\line(0,-1){40}}
	\put(155,-40){\line(-1,0){20}}
	\put(260,-115){\makebox(0,0)[tl]{$A_4$}}
	\put(225,-40){\line(-1,0){20}}
	\put(205,-80){\line(0,1){40}}
	\put(205,-80){\line(1,0){20}}
	\put(225,-100){\line(0,1){20}}
	\put(225,-100){\line(1,0){80}}
	\put(305,-80){\line(0,-1){20}}
	\put(305,-80){\line(1,0){20}}
	\put(325,-40){\line(0,-1){40}}
	\put(325,-40){\line(-1,0){20}}
	\put(225,-80){\line(-1,-1){10}}
	\put(215,-90){\line(-1,1){10}}
	\put(205,-80){\line(-1,1){10}}
	\put(195,-50){\line(1,1){10}}
	\put(205,-60){\line(-1,1){10}}
	\put(195,-70){\line(1,1){10}}
	\put(205,-40){\line(1,1){10}}
	\put(215,-90){\line(1,-1){10}}
	\put(215,-30){\line(1,-1){10}}
	\put(225,-100){\line(1,-1){10}}
	\put(235,-110){\line(1,1){10}}
	\put(245,-100){\line(1,-1){10}}
	\put(255,-110){\line(1,1){10}}
	\put(265,-100){\line(1,-1){10}}
	\put(265,-100){\line(0,1){20}}
	\put(265,-100){\line(1,1){10}}
	\put(275,-90){\line(-1,1){10}}
	\put(265,-80){\line(-1,-1){10}}
	\put(255,-90){\line(1,-1){10}}
	\put(275,-110){\line(1,1){10}}
	\put(285,-100){\line(1,-1){10}}
	\put(295,-110){\line(1,1){30}}
	\put(305,-40){\line(1,1){10}}
	\put(325,-60){\line(1,-1){10}}
	\put(335,-70){\line(-1,-1){10}}
	\put(325,-40){\line(1,-1){10}}
	\put(335,-50){\line(-1,-1){10}}
	\put(315,-30){\line(1,-1){10}}
	\put(315,-90){\line(-1,1){10}}
	\end{picture}
\caption{The first five steps of the construction of L\'evy's dragon curve}
\label{fig:construction}
\end{center}
\end{figure}

A natural question that arises is whether there are more ways to express L\'evy's Dragon Curve $L$; more specifically, does an explicit representation exist? In fact, any point $z$ of $L$ can be represented by at least one infinite sequence $(x_i)_{i=1}^{\infty} \in \{0,1\}^{\NN}$ such that 
\begin{equation*}
    z= \lim_{n \to \infty}\psi_{x_1} \circ \psi_{x_2} \circ \dots \circ \psi_{x_n} (0).
\end{equation*}


In 2002, approaching the curve $L$ thought the lens of functional equation, Kawamura~\cite{Kawamura-2002} showed that any point $z$ of $L$ can be expressed as a complex power series. That is, 
\begin{equation*}
z= \sum_{n=1}^{\infty} \xi_{n}(1+i)^{-n}, 
\end{equation*}
where the infinite sequence $(\xi_n)_{n=1}^{\infty} \in \{0,1,i,-1,-i\}^{\NN}$ with the restriction that all non-zero values of $\xi_n$ must follow the unit circle in a counter clockwise fashion, starting at 1. In other words the non-zero digits follow Figure~\ref{fig:revdig}.
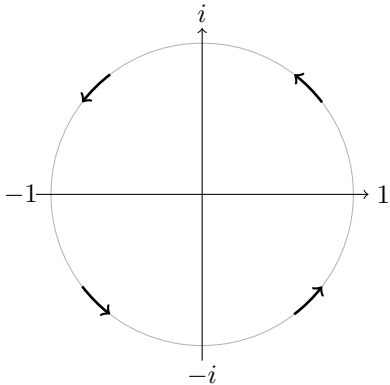
\begin{figure}[H]
    \begin{center}
\begin{tikzpicture}[scale=2]
    \draw[->] (-1.1,0) -- (1.1,0);
    \draw[->] (0,-1.1) -- (0,1.1);
    \draw [opacity=0.3](0,0) circle (1);
    \foreach \angle in {37.5, 127.5, 217.5, 307.5} {
        \draw[->,line width=1pt] (\angle:1) arc (\angle:\angle+15:1);
    }
    \node at (1.2,0) {$1$};
    \node at (-1.2,0) {$-1$};
    \node at (0,1.2) {$i$};
    \node at (0,-1.2) {$-i$};
\end{tikzpicture}
\end{center} 
    \caption{The unit circle in the complex plane with rotation $\theta=\frac{\pi}{2}$}
    \label{fig:revdig}
\end{figure}

So we have shown that L\'evy's Dragon Curve can be expressed in a variety of different ways, each of which provide insight into the behavior of the curve. Then following question naturally arises.\medskip

Can any translation of Le\'vy's Dragon Curve be expressed in similar manner to the original Le\'vy's Dragon Curve? 
\medskip

In this paper, as the first step, we investigate the translation of L\'evy's Dragon Curve by $s=-1/2+i/2$.

\section{Self-similar Sets and Functional Equations}

The history of systematic mathematical research on self-similar sets dates back to 1981, when Hutchinson considered the non-empty compact set $X \subset \mathbb{R}^{n}$ satisfying the following set equation.
\begin{equation}
 \label{eq:self-similar}
X= \varphi_{0}(X) \cup \varphi_{1}(X) \cup \dots \cup \varphi_{m-1}(X),
\end{equation}
where $\varphi_{0},\varphi_{1},\dots,\varphi_{m-1}$ are similarity  contractions on $\mathbb{R}^{n}$. 

(Recall that a map $\psi:\mathbb{R}^{n} \rightarrow \mathbb{R}^{n}$ is a {\it similarity contraction} iff there exists a constant number $L(\varphi) \in (0,1)$ so that the equality $\|\varphi(x)-\varphi(y)\|=L(\varphi)\|x-y\|$ holds for any $x,y \in \mathbb{R}^{n}$).

Hutchinson~\cite{Hutchinson-1981} proved the following important theorem. 
\begin{theorem}[Hutchinson (1981)]
\label{th:hutchinson}
For any finite family of similarity contractions $\{\varphi_0, \varphi_1, ..., \varphi_{m-1}\}$ in $\mathbb{R}^{n}$, which is called an iterated function system (or IFS), there exists a unique non-empty compact solution $X$ of \eqref{eq:self-similar} in $ \mathbb{R}^{n}$.
\end{theorem} 

We call $X$ an attractor or a self-similar set for a given IFS. 

It is well-known that any point of $X$ can be represented by at least one coding sequence $(x_i)_{i=1}^{\infty}$ such that 
\begin{equation*}
X= \left\{\lim_{n \to \infty} \varphi_{x_1} \circ \varphi_{x_2} \circ \cdots \varphi_{x_n} (0): (x_i)_{i=1}^{\infty} \in \{0,1,2, \cdots (m-1)\}^{\mathbb{N}}\right\}
\end{equation*}

In this paper, we focus on a self-similar set $X$ determined by IFS with 2 similarity contractions $\{\psi_0,\psi_1\}$.

\bigskip

In 1957, G. ~de Rham studied the following general functional equation and showed the following result.  

\begin{equation}
 \label{func: gene}
        f(x)=
        \begin{cases}
                 \varphi_{0}(f(2x)),
                        & \qquad 0 \leq x < 1/2, \\
                 \varphi_{1}(f(2x-1)),
                        & \qquad 1/2 \leq x \leq 1, \\
        \end{cases}
\end{equation}
where $\varphi_{0}, \varphi_{1}$ are contractions on $\mathbb{R}^2$.

\begin{theorem}[de Rham (1957)]
There exists a unique continuous solution $f(x)$ of \eqref{func: gene} if and only if 
$$\varphi_{1}(Fix (\varphi_0)) = \varphi_{0}(Fix(\varphi_1)),$$ 
where $Fix (\varphi_0), Fix (\varphi_1)$ are the unique fixed points of $\varphi_{0}$ and $\varphi_{1}$, respectively. We call $x$ a fixed point of $\varphi$ if $\varphi(x)=x$, and we denote this as $x= Fix(\varphi)$.
\end{theorem}

Recall that L\'evy's Dragon Curve $L$ is a self-similar set on $\mathbb{C}$, since $L$ is a unique attractor generated by \eqref{eq:levy}.

Consider the following functional equation:
\begin{equation}
    \label{eq:deLevy}
        f(x)=
        \begin{cases}
                \alpha f(2x),
                        & \qquad 0 \leq x < 1/2, \\
                (1-\alpha) f(2x-1)+\alpha,
                        & \qquad 1/2 \leq x \leq 1,
       \end{cases}
\end{equation}
where $\alpha=(1-i)/2$.

Notice that \eqref{eq:deLevy} is a special case of de Rham's functional equation \eqref{func: gene} so then by applying de Rham's theorem we have the following corollary.

\begin{corollary}
There exists a unique continuous solution $f(x)$ of \eqref{eq:deLevy}.
\end{corollary}
\begin{proof}
Let $\psi_0 = \alpha f(2x)$ and $\psi_1 =(1-\alpha) f(2x-1)+\alpha$, with $\alpha = (1-i)/2$. Observe that $Fix (\psi_0)\text{ and } Fix (\psi_1)$ are 0 and 1 respectively. Then 
\begin{align*}
    \psi_{1}(Fix (\psi_0)) = \psi_1(0)=\alpha \\
    \psi_{0}(Fix(\psi_1))=\psi_0(1) =\alpha
\end{align*}

Therefore, by Theorem 2.2 there exists an unique continuous solution of \eqref{eq:deLevy}.
\end{proof}

Observe that the unique continuous solution of \eqref{eq:deLevy} is a complex-valued function $f: [0,1]\to \mathbb{C}$ and the image $f([0,1])$ is L\'evy's dragon curve $L$. In other words, $\{f(x);0\leq x \leq 1\}$ is a parametrization of L\'evy's dragon curve $L$.

\begin{figure}[H]
  	\begin{center}
	\epsfig{file=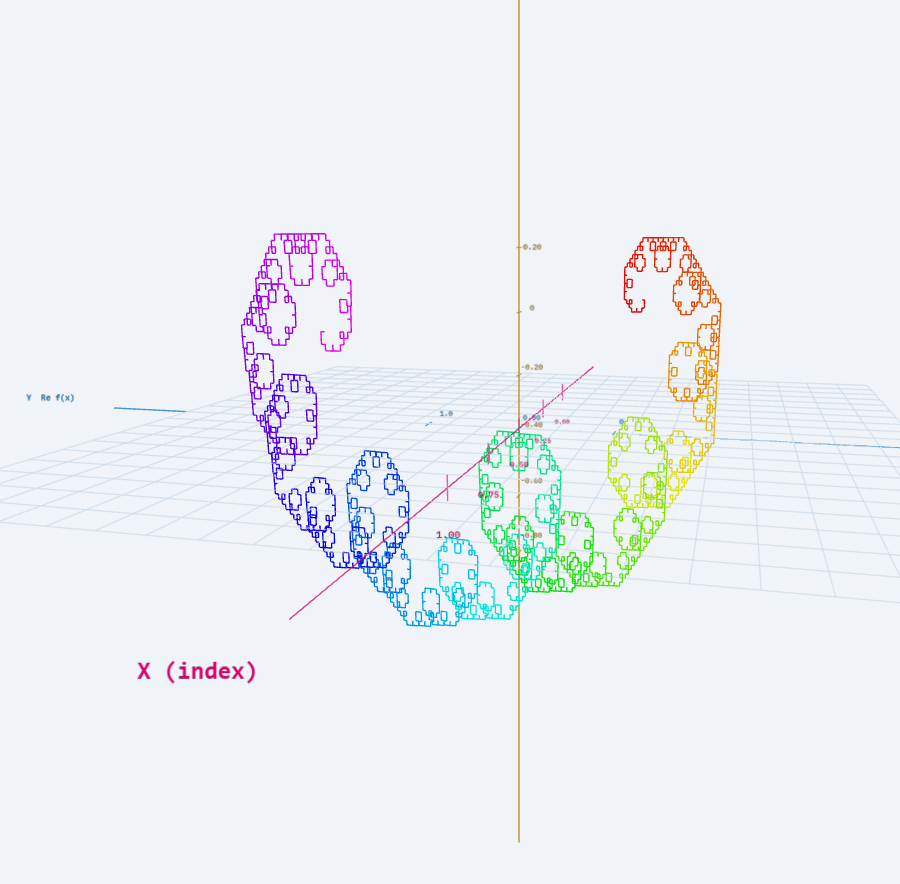, width=.65\textwidth} 
	\end{center}
     \caption{The graph of the complex-valued function $f(x)$. The color corresponds to the x-axis, which is the index set [0,1], with pink at 0, and red at 1. The y-axis and z-axis, are the real and imaginary components of $f(x)$ respectively.}
    \label{fig:3D-fractal}
\end{figure}
\section{Revolving Representation of L\'evy's Dragon Curve}

First, we define the binary expansion of $x \in [0,1]$ as follows. 
\begin{definition}

The binary expansion of $x \in [0,1]$ is denoted by  
\begin{equation}
    x= \sum_{n=1}^{\infty}\omega_n2^{-n}=(0.\omega_1 \omega_2 \cdots )_{2} \qquad \omega_n = \omega_n(x) \in \{0,1\}.
\end{equation}

For $x\in [0,1]$ with two binary expansions, choose the expansion with trailing zeros. However, if $x=1$, fix $\omega_n(x) =1$ for every $n$.

Let $q(x,n) = \sum_{k=1}^n \omega_k$. Equivalently, $q(x,n)$ is the number of $1$'s occurring in the first $n$ binary digits of $x$. By convention, $q(x,0)=0$.

\end{definition}

In 2002, Kawamura studied the functional equation \eqref{eq:deLevy} where $\alpha$ is a complex parameter satisfying $|\alpha|<1$, and gave the following explicit formula for the continuous solution of \eqref{eq:deLevy}.

\begin{theorem}[Kawamura, 2002~\cite{Kawamura-2002}]
The unique continuous solution of \eqref{eq:deLevy} has the following expression.
\begin{equation}
    \label{kiko-levy}
    f(x) = \sum_{n=1}^{\infty} \omega_n(x)\alpha^{n-q(x,n-1)}(1-\alpha)^{q(x,n-1)}, \qquad 0 \leq x \leq 1.
\end{equation}
\end{theorem}

Kawamura also showed that if $\alpha=(1-i)/2$, $f(x)$ can be expressed as a complex power series restricted by a special \textit{revolving condition}. That is, if $\alpha=(1-i)/2$, then $\alpha^{-1}=(1+i)$ and 
\begin{equation}
\label{kiko-levy-2}
    f(x) = \sum_{n=1}^{\infty} \omega_n(x)\alpha^n(i)^{q(x,n-1)}=\sum_{n=1}^{\infty}\omega_n(x)(1+i)^{-n}(i)^{q(x,n-1)}
    \qquad 0 \leq x \leq 1.
\end{equation}

Let $\xi_n:=\omega_n(i)^{q(x,n-1)}$. Observe that $\xi_n \in \{0, 1,-1, i,-i\}$ with the restriction that the non-zero values must follow the cyclic pattern from left to right: 
$$1  \to i \to (-1) \to (-i) \to 1 \to \cdots.$$
The sequence $(\xi_n)$ varies depending on $x \in [0,1]$. Figure~\ref{fig:map1} shows the directed-graph $\GG_1$ of $(\xi_n)$. In Figure~\ref{fig:revdig} there is no characterization of the 0 digit, however it maintains a memory function in which the number preceding the 0 restricts the following digit. To describe this difference we create the following directed graph $\GG_1$:

\begin{figure}[H]
  	\begin{center}
	\epsfig{file=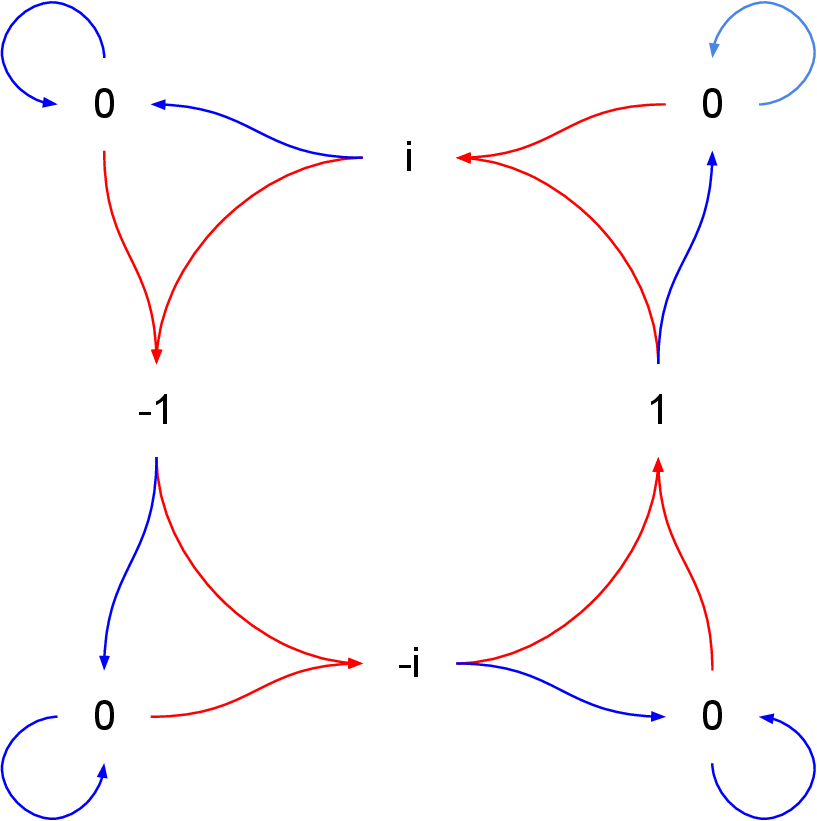, height=1.75in, width=.51\textwidth} 
	\end{center}
    \caption{The directed-graph $\GG_1$, that determines the points of the L\'evy's Dragon Curve $L$}
\label{fig:map1}
\end{figure}

\bigskip

In 2021, Kawamura and Allen~\cite{Kawamura-Allen-2021} introduced the definition of \textit{Generalized Revolving Condition} as follows.

Let $\theta$ be an angle with $-\pi < \theta \leq \pi$ and a rational multiple of $2 \pi$. More precisely, there are $p \in \NN, q \in \NN_{0}$ such that $|\theta|=\frac{2 \pi q}{p}$. 

Define 
$$\Delta_{\theta}:= \{0, 1, e^{i \theta}, e^{2i \theta}, \cdots e^{(p-1)i \theta}\}.$$

\begin{definition}
A sequence $(\delta_{1},\delta_{2},\dots)\in\Delta_{\theta}^{\NN}$ satisfies the {\em Generalized Revolving Condition (GRC)}, if the  subsequence obtained after the removal of its zero elements is a (finite or infinite) truncation of the sequence $(e^{i \theta})$. More precisely, let $(n_i):=\{n:\delta_n \neq 0\}$. Then, 
$\delta_{n_{i+1}}=e^{i \theta} \delta_{n_i}.$
\end{definition}

Notice that $\delta_n$ moves on the unit circle counterclockwise if $\theta >0$, and clockwise if $\theta < 0$. 
\bigskip

Fix $\theta = \frac{\pi}{2}$ so that $\Delta= \{0,1,i,-1,-i\}$.
Observe that any sequence $(\xi_n)\in \Delta^{\NN}_{\frac{\pi}{2}}$ satisfies the GRC. Since $L=f([0,1])$, we have 
\begin{equation} 
    \label{L-revolving}
  L=\left\{ \sum_{n=1}^{\infty} \xi_{n}\alpha^{n}: (\xi_n) \mbox{ follows the directed-graph $\GG_1$} \right\}, 
\end{equation}
where $\alpha=(1-i)/2.$ That is, any point $z$ of L is characterized by a sequence $(\xi_i)_{i=1}^{\infty} $ satisfying the GRC with $\theta=\frac{\pi}{2}$.

\section{Main Results}

We investigate the translation of $L$ by $s=-1/2+i/2$ and call this the s-shifted L\'evy's Dragon Curve: $L_s$. $L_s$ inherits the self-similarity of $L$, so it is also the attractor of an IFS $\{\phi_0, \phi_1\}$. Let $S(z)=z-\frac{1-i}{2}$ on $\mathbb{C}$.
We obtain these functions by modifying $\psi_i$ in the following way:
\begin{equation*}
\phi_i = S \circ \psi_i \circ S^{-1}(z)
\end{equation*}
Thus, $L_s$ is the unique attractor $L_s=\phi_0(L_s)\cup\phi_1(L_s)$, which is  generalized by the following modified IFS: 
\begin{equation}
\label{eq:s-levy}
\begin{cases}
 \phi_0(z)=(\frac{1-i}{2})(z-s)+s = (\frac{1-i}{2}) z-\frac{1}{2},& \\
 \phi_1(z)=(\frac{1+i}{2})(z-s) + (\frac{1-i}{2})+s = (\frac{1+i}{2}) z +\frac{1}{2}&
\end{cases}
\end{equation}

Then a natural question arises: Can any point $z$ of $L_s$ be expressed as a complex power series with a similar sequence satisfying the GRC?

\medskip
Consider the following functional equation
\begin{equation}
\label{eq:s-functional}
 G(x) =
\begin{cases}
 \alpha G(2x) - \frac{1}{2}
                     &\qquad 0 \leq x < \frac{1}{2}, \\
(1-\alpha) G(2x-1) + \frac{1}{2} & \qquad \frac{1}{2} \leq x \leq 1,\\
 
 \end{cases}
 \end{equation}
where $\alpha=(1-i)/2$.
Notice that \eqref{eq:s-functional} is similar, but not equivalent to \eqref{eq:deLevy}. It is also another special case of de Rham's functional equation \eqref{func: gene}. Therefore, using Theorem 2.2, we have 
 \begin{lemma}
There exists a unique continuous solution $G(x)$ of \eqref{eq:s-functional}.
\end{lemma}
\begin{proof}
Every contraction $\phi_i:\mathbb{C} \to \mathbb{C}$ has a unique fixed point $Fix(\phi_i)$ in $\mathbb{C}$. Since $Fix(\phi_0) =\frac{1}{2(\alpha-1)}= \frac{-1+i}{2}$ and $Fix(\phi_1) = \frac{1}{2\alpha}=\frac{1+i}{2}$, observe that 
$$\phi_1(Fix(\phi_0)) = 0 = \phi_0(Fix(\phi_1)).$$

\end{proof}
\bigskip
Now, the main interest is to find an explicit expression of $G(x)$. Observe that from \eqref{eq:s-functional}, we have
\begin{itemize}
\item If $x=0$, $G(0) = \alpha G(0) - 1/2 \implies G(0) = \frac{-1}{2(1-\alpha)} =\frac{-1+i}{2}$.

\item If $x=1$, $G(1) = (1-\alpha) G(1) + 1/2 \implies G(1) =\frac{1}{2\alpha}=\frac{1+i}{2}$.

\item If $x=1/2$, $G(1/2) = (1-\alpha) G(0) + 1/2 \implies G(1/2) =0$.

\item If $x=1/4$, $G(1/4) = \alpha G(1/2) - 1/2 \implies G(1/4)=-1/2$.

\item If $x=3/4$, $G(3/4) = (1-\alpha) G(1/2) + 1/2 \implies G(3/4) =1/2$.
\end{itemize}
Inductively, we can calculate $G(x)$ for all dyadic points $x \in [0,1]$. 

\begin{lemma}
\label{th:shiftedlevyexplicit}
For any dyadic point $x \in (0,1)$; that is, there exists $k\in \mathbb{N}$ such that $\omega_k=1$ and $\omega_n(x)=0$ for $\forall n> k$, $G(x)$ has the following expression. 
\begin{equation}
    \label{explicts-levy}
    G(x) = \frac{1}{2}\sum_{n=1}^{k-1}(-1)^{1-\omega_n(x)}\alpha^{n-1-q(x,n-1)}(1-\alpha)^{q(x,n-1)}
\end{equation}
where $\alpha$ = $(1-i)/2$.
\end{lemma} 

\begin{proof}[Proof of Lemma \ref{th:shiftedlevyexplicit}]

Let 
\begin{equation*}
F(x) := \frac{1}{2}\sum_{n=1}^{k-1}(-1)^{1-\omega_n(x)}\alpha^{n-1-q(x,n-1)}(1-\alpha)^{q(x,n-1)}.
\end{equation*}

Recall that $G(x)$ is the unique continuous solution of \eqref{eq:s-functional}.

Let $P(k)$ be the statement: 
$$F(x)=G(x) \mbox{ for all } x \mbox{ of the form } x=(0.\omega_1 \cdots .\omega_{k-1} 1)_{2}.$$

First, consider the base case $k=1$. Since $x=(0.1)_{2}=1/2$, it is clear that the statement $P(1)$ holds as  
\begin{align*}
F(1/2)=F((0.1)_{2}) &= \frac{1}{2}\sum_{n=1}^{0}(-1)^{1- \omega_n(x)}\alpha^{n-1-q(x,n-1)}(1-\alpha)^{q(x,n-1)}\\
& = 0 =G(1/2). 
\end{align*}

Second, consider the case $k=2$, that is $x=(0.01)_2=1/4$ or $x=(0.11)_2= 3/4$

\begin{align*}
F(1/4)=F((0.01)_{2}) &= \frac{1}{2}(-1)^{1-0}\cdot \alpha^0(1-\alpha)^0 =-\frac{1}{2} =G(1/4). \\
F(3/4)=F((0.11)_{2}) &= \frac{1}{2}(-1)^{1-1}\cdot \alpha^0(1-\alpha)^0 =\frac{1}{2} =G(3/4). 
\end{align*}
Therefore the statement P(2) holds.

We continue the proof by induction. Let $l \geq 2$ and assume that the statement $P(l)$ is true; that is,  
$$F(x) = G(x) \mbox { for all } x \mbox{ of the form } x=(0.\omega_1\omega_2\omega_3 \dots \omega_{l-1}1)_2.$$ 
We want to prove $P(l+1)$.

\medskip

Consider $x$ having the form $x=(0.\omega_1\omega_2\omega_3 \dots \omega_{l}1)_2$.
Notice that for $0\leq x<1/2$, $\omega_1=0$ and $q(2x,n-1) = q(x,n)$ and for $1/2 \leq x \leq 1$, $\omega_1=1$ and $q(2x-1,n-1) = q(x,n)-1$.\\
If $0\leq x< 1/2$,
\begin{align*}
    F(x) &= F(0.0 \omega_2\dots\omega_l1)_2\\
    &= \frac{1}{2}\sum_{n=1}^l(-1)^{1-\omega_n(x)}\alpha^{n-1-q(x,n-1)}(1-\alpha)^{q(x,n-1)}\\
    &=-\frac{1}{2}+\frac{1}{2}\sum_{n=2}^l(-1)^{1-\omega_n(x)}\alpha^{n-1-q(x,n-1)}(1-\alpha)^{q(x,n-1)}\\
    &=-\frac{1}{2}+\frac{1}{2}\sum_{n=1}^{l-1}(-1)^{1-\omega_{n+1}(x)}\alpha^{n-q(x,n)}(1-\alpha)^{q(x,n)}\\
    &=-\frac{1}{2}+\frac{\alpha}{2}\sum_{n=1}^{l-1}(-1)^{1-\omega_{n}(2x)}\alpha^{n-1-q(2x,n-1)}(1-\alpha)^{q(2x,n-1)}\\
    &= -\frac{1}{2}+\alpha F(2x)
\end{align*}
Then observe that $2x = (0.\omega2\omega3\omega4\dots\omega_{l-1}1)$.
Therefore $P(l)$ holds by assumption,
\begin{align*}
   F(x) &= -\frac{1}{2}+\alpha G(2x)= G(x)
\end{align*}
If $1/2 \leq x\leq1$,
\begin{align*}
    F(x) &= F(0.1 \omega_2\dots\omega_l1)_2\\
    &= \frac{1}{2}\sum_{n=1}^l(-1)^{1-\omega_n(x)}\alpha^{n-1-q(x,n-1)}(1-\alpha)^{q(x,n-1)}\\
    &=\frac{1}{2}+\frac{1}{2}\sum_{n=2}^l(-1)^{1-\omega_n(x)}\alpha^{n-1-q(x,n-1)}(1-\alpha)^{q(x,n-1)}\\
    &=\frac{1}{2}+\frac{1}{2}\sum_{n=1}^{l-1}(-1)^{1-\omega_{n+1}(x)}\alpha^{n-q(x,n)}(1-\alpha)^{q(x,n)}\\
    &=\frac{1}{2} +\frac{1}{2}\sum_{n=1}^{l-1}(-1)^{1-{\omega_n(2x-1)}}\alpha^{n-q(2x-1,n-1)-1}(1-\alpha)^{q(2x-1,n-1)+1}\\
    &=\frac{1}{2}+\frac{(1-\alpha)}{2}\sum_{n=1}^{l-1}(-1)^{1-\omega_n(2x-1)}\alpha^{n-1-q(2x-1,n-1)}(1-\alpha)^{q(2x-1,n-1)}\\
    &=\frac{1}{2}+(1-\alpha)F(2x-1)
\end{align*}
Again, observe that $2x = (0.\omega2\omega3\omega4\dots\omega_{l-1}1)$.
Therefore $P(l)$ holds by assumption, we have 
\begin{align*}
   F(x) &= \frac{1}{2}+(1-\alpha) G(2x-1)= G(x)
\end{align*}
Therefore, we conclude that the statement $P(k)$ is true for all $k \in \mathbb{N}$.
\end {proof}
\bigskip
Recall $\exists k\in \mathbb{N} \text { such that } \omega_k=1$ and $\omega_n=0$ for $\forall n> k$, if $x$ is a dyadic point. Therefore, $q(x,k-1)+1 = q(x,k)  =q(x, k+l)$ for $\forall l \in \NN$. Let
    \begin{equation*}
        H(x): = \frac{1}{2} \sum_{n=k}^{\infty}(-1)^{1-\omega_n}\alpha^{n-1-q(x,n-1)}(1-\alpha)^{q(x,n-1)}
    \end{equation*} 
    \begin{align*}
       H(x)
         &=\frac{1}{2}\left((-1)^0 \alpha^{k-1}\left(\frac{1-\alpha}{\alpha}\right)^{q(x,k-1)}+\sum_{n=k+1}^{\infty}(-1)^{1-\omega_n}\alpha^{n-1}\left(\frac{1-\alpha}{\alpha}\right)^{q(x,n-1)}\right)\\
         &=\frac{1}{2}\left( \alpha^{k-1}\left(\frac{1-\alpha}{\alpha}\right)^{q(x,k-1)}+(-1)^{1}\left(\frac{1-\alpha}{\alpha}\right)^{q(x,k-1)+1}\sum_{n=k+1}^{\infty}\alpha^{n-1}\right)\\
         &=\frac{1}{2}\left( \alpha^{k-1}\left(\frac{1-\alpha}{\alpha}\right)^{q(x,k-1)}-\left(\frac{1-\alpha}{\alpha}\right)^{q(x,k-1)+1} \left( \frac{\alpha^k}{1-\alpha}\right)\right)\\
         &=\frac{1}{2}\left(\alpha^{k-1}\right) \left(\frac{1-\alpha}{\alpha}\right)^{q(x,k-1)}\left(1-\left(\frac{1-\alpha}{1}\right)\left(\frac{1}{1-\alpha}\right)\right)\\
         &=0
    \end{align*}
Therefore for a dyadic point $x\in[0,1]$
\begin{equation*}
    G(x) = \frac{1}{2}\sum_{n=1}^{\infty}(-1)^{1-\omega_n(x)}\alpha^{n-1-q(x,n-1)}(1-\alpha)^{q(x,n-1)}.
\end{equation*}
Since the set of dyadic points is dense in [0,1] and $G(x)$ is the unique continuous solution, we have the following theorem.
\begin{theorem}
    \label{Main-Result}
    The unique continuous solution of G(x) of \eqref{eq:s-functional} can be expressed as follows.
    \begin{equation}
    G(x) = \frac{1}{2} \sum^\infty_{n=1} (-1)^{1-\omega_n(x)} \alpha^{n-1-q(x,n-1)}(1-\alpha)^{q(x,n-1)} \qquad x \in [0,1] 
    \end{equation}
    Where $\alpha = (1-i)/2$.
\end{theorem}
\bigskip

Finally, recall our main question, Is any point $z$ of $L_s=G([0,1])$  also characterized by a sequence satisfying the GRC?

Using \eqref{kiko-levy-2} as a model, we have 
\begin{align*}
    G(x) &= \frac{1}{2} \sum^\infty_{n=1} (-1)^{1-\omega_n(x)} \alpha^{n-1-q(x,n-1)}(1-\alpha)^{q(x,n-1)} \\
    &=\frac{1}{2\alpha}\sum^\infty_{n=1} (-1)^{1-\omega_n(x)} \alpha^{n}(i)^{q(x,n-1)} \\
    &=(1-\alpha)\sum^\infty_{n=1} \alpha^n (-1)^{1-\omega_n(x)}(i)^{q(x,n-1)}. 
    \end{align*}
    
Let $\gamma_n:=(-1)^{1-\omega_n(x)}(i)^{q(x,n-1)}$. Observe that $\gamma_n \in \{1,-1, i,-i\}$, which does not include $0$. $(\gamma_n)$ moves on the unit circle, but the movement varies depending on $x \in [0,1]$.

For example, consider the binary expansion of $x=(0.100110110)_2$. Then the corresponding digit sequence $(\gamma_n)$ is 
$(\gamma_n)=(1, -i, -i, i, -1, i, -i, 1, -i)$
It is clear that any sequence $(\gamma_n) \in \{1,-1, i,-i\}^{\mathbb{N}}$ does not satisfy the GRC. 

Loosely speaking, we can say that the sequence $(\gamma_n)\in \{1,-1, i,-i\}^{\mathbb{N}}$ is a sequence on the unit circle and with subsequent terms which must perform one of the following actions:  
\begin{itemize}
\item Stay in place 
\item Move 1 step or 2 steps forward (counterclockwise) on the unit circle
\item move 1 step backward (clockwise) on the unit circle
\end{itemize}

Initially, $(\gamma_n)$ is seemingly much more complicated than the sequence $(\xi_n)$. In fact, it is rather difficult to find the similarity between the revolving structures of the sequences $(\xi_n)$ and  $(\gamma_n)$. 

However, from the viewpoint of the directed graph, the similarity is clear. We see that the sequence $(\gamma_n)\in \{1,-1, i,-i\}^{\mathbb{N}}$ follows the directed-graph $\GG_2$ shown in Figure~\eqref{fig:map2}.

\begin{figure}[H]
  	\begin{center}
	\epsfig{file=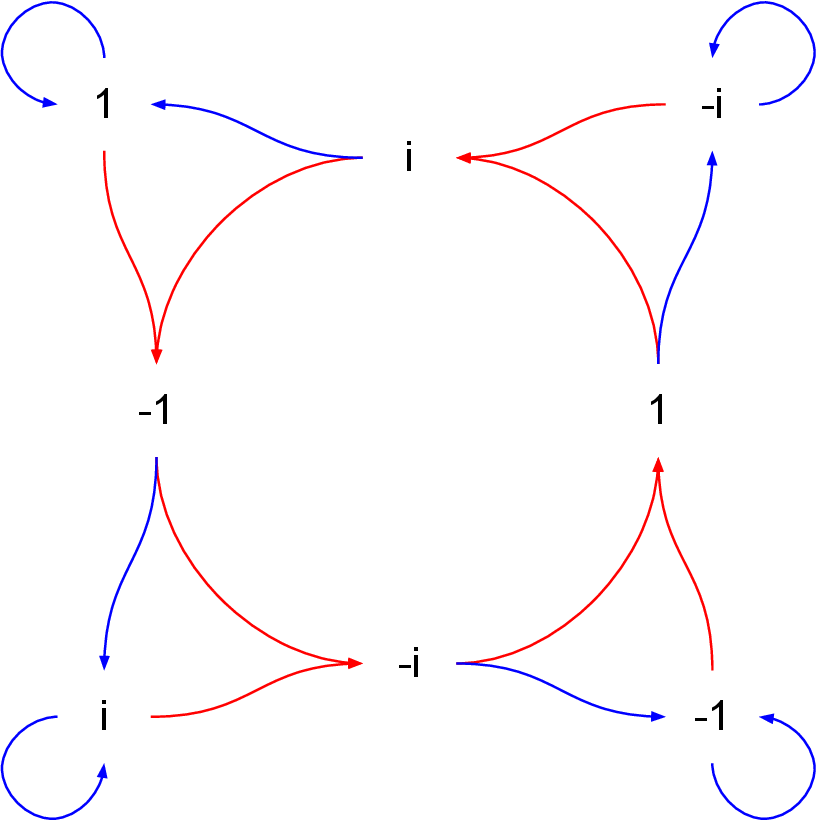, height=1.75in, width=.51\textwidth} 
	\end{center}
    \caption{The directed-graph $\GG_2$, that determines the points of the Shifted L\'evy's Dragon Curve $L_s$}
\label{fig:map2}
\end{figure}

Another use of this directed graph it to generate the digit sequence $(\gamma_n)$ directly from the binary sequence of $x$ without computation. The rules of construction are as follows: If $\omega_n = 0$ then the sequence follows the blue arrow from the previous digit; If $\omega_n=1$ then the sequence follows the red arrow from the previous digit. 

Observe that $\gamma_1=-1$ if $\omega_1=0$ and $\gamma_1=1$ if $\omega_1=1$. We interpret $\gamma_1$ as specifying the initial position in Figure~\ref{fig:map2}: The value $-1$ corresponds to the point labeled ``$-1$'' in the lower right region, and the value $1$ corresponds to the point labeled ``$1$'' on the right of the inner circle.

For example, consider the binary expansion of $x=(0.01011001101)_2$. Then the corresponding digit sequence $(\gamma_n)$ starts with $\gamma_1 = -1$ because $\omega_1=0$. Then, since $\omega_2=1$, following the red arrow, we find $\gamma_2=1$. Since $\omega_3=0$, following the blue arrow, we find $\gamma_3=-i$. 
Continue this process, then we obtain:
$(\gamma_n)=(-1,1, -i, i, -1, i, i,-i, 1, -i,i)$.

\medskip

Since $L_s=G([0,1])$, we have 
\begin{equation}
\label{eq: L2 expression}
  L_s=\left\{ (1-\alpha) \sum_{n=1}^{\infty} \gamma_{n}\alpha^{n}: (\gamma_{n} ) \mbox{ follows the directed-graph $\GG_2$}\} \right\}, 
\end{equation}

Compare with Figure~\ref{fig:map1} and Figure~\ref{fig:map2}. Observe that the directed-graph $\GG_2$ for $L_s$ looks identical to the directed-graph $\GG_1$ for $L$, with only the labels of the nodes being different. Therefore, we can say that both the original L\'evy dragon curve $L$ and 
the s-shifted L\'evy's dragon curve $L_s$ have the similar revolving structure. 

\section{Future work}

In this paper, we investigated a particular shifted L\'evy's dragon curve $L_s$ and introduced the directed-graph $\GG_2$, which characterizes $L_s$. 

A natural next step is to examine the behavior under more general transformations. Specifically, let $\lambda$ be an arbitrary scaling factor and $\tau$ an arbitrary translation. Consider the set obtained by scaling the original L\'evy's dragon curve $L$ 
by $\tau$ and then scale by $\lambda$. Call the resulting set $L_{\lambda,\tau}$. It is straightforward to verify that $L_{\lambda,\tau}$ is the unique attractor satisfying  $L_{\lambda,\tau}=g_0(L_{\lambda,\tau})\cup g_1(L_{\lambda,\tau})$, 
where the IFS is given by  
\begin{equation*}
\label{eq:general shifted levy}
\begin{cases}
 g_0(z)= (\frac{1-i}{2}) z + (\frac{1+i}{2})\tau, \\
 g_1(z)=(\frac{1+i}{2}) z+ (\frac{1-i}{2}) (\tau + \lambda).
\end{cases}
\end{equation*}

Furthermore, this naturally raises the question of whether the attractor $L_{\lambda,\tau}$ admits an explicit representation analogous to \eqref{eq: L2 expression}. In particular, one may ask whether $L_{\lambda,\tau}$ can be described via a modified sequence governed by a directed graph
$\GG_3$, whose structure is analogous to those of $\GG_1$ and $\GG_2$.

\section*{Acknowledgments}
The author would like to thank their thesis advisor, Dr. Kiko Kawamura, for helpful feedback and revisions. The author also thanks Dr. R. Daniel Prokaj and Miguel Gonzalez-Carriedo for their assistance in the creation of the central theorem. Finally, the author thanks Soham Mangesh Kale and Anshika Alok Pandey for their 3D rendering of Figure~\ref{fig:3D-fractal}. 


\end{document}